\documentclass[12pt]{amsart}
\usepackage{amscd}
\usepackage{graphicx}
\input xy
\xyoption{all}
\def\frk{\mathfrak}               

\def\Ff{{\frk F}}

\def\Phi{{\frk N}}
\def\opn#1#2{\def#1{\operatorname{#2}}} 
\opn\chara{char} \opn\length{\ell} \opn\pd{pd} \opn\rk{rk}
\opn\projdim{proj\,dim} \opn\injdim{inj\,dim} \opn\rank{rank}
\opn\depth{depth} \opn\grade{grade} \opn\height{height}
\opn\embdim{emb\,dim} \opn\codim{codim}

\opn\Tr{Tr} \opn\bigrank{big\,rank}
\opn\superheight{superheight}\opn\lcm{lcm}
\opn\trdeg{tr\,deg}
\opn\reg{reg} \opn\lreg{lreg} \opn\ini{in} \opn\lpd{lpd}
\opn\size{size}\opn{\mult}{mult}
\opn\div{div} \opn\Div{Div} \opn\cl{cl} \opn\Cl{Cl}
\opn\Spec{Spec} \opn\Supp{Supp} \opn\supp{supp} \opn\Sing{Sing}
\opn\Ass{Ass} \opn\Min{Min}
\opn\Ann{Ann} \opn\Rad{Rad} \opn\Soc{Soc}
\opn\Syz{Syz} \opn\Im{Im} \opn\Ker{Ker} \opn\Coker{Coker}
\opn\Am{Am} \opn\Hom{Hom} \opn\Tor{Tor} \opn\Ext{Ext}
\opn\End{End} \opn\Aut{Aut} \opn\id{id} \opn\ini{in}

\opn\nat{nat}
\opn\pff{pf}
\opn\Pf{Pf} \opn\GL{GL} \opn\SL{SL} \opn\mod{mod} \opn\ord{ord}
\opn\Gin{Gin}
\opn\Hilb{Hilb}\opn\adeg{adeg}\opn\std{std}\opn\ip{infpt}
\opn\Pol{Pol}
\opn\sat{sat}
\opn\Var{Var}
\opn\Gen{Gen}

\opn\aff{aff} \opn\con{conv} \opn\relint{relint} \opn\st{st}
\opn\lk{lk} \opn\cn{cn} \opn\core{core} \opn\vol{vol}
\opn\link{link} \opn\star{star}
\opn\gr{gr}

\def\Ic{{\mathcal I}}

\def\Fc{{\mathcal F}}
\def\Nc{{\mathcal N}}
\def\Mc{{\mathcal M}}

\def\pot#1#2{#1[\kern-0.28ex[#2]\kern-0.28ex]}

\opn\dirlim{\underrightarrow{\lim}}
\opn\inivlim{\underleftarrow{\lim}}
\def\Implies{\ifmmode\Longrightarrow \else
        \unskip${}\Longrightarrow{}$\ignorespaces\fi}
\def\implies{\ifmmode\Rightarrow \else
        \unskip${}\Rightarrow{}$\ignorespaces\fi}
\def\iff{\ifmmode\Longleftrightarrow \else
        \unskip${}\Longleftrightarrow{}$\ignorespaces\fi}

\let\:=\colon
\newtheorem{Theorem}{Theorem}[section]
\newtheorem{Lemma}[Theorem]{Lemma}
\newtheorem{Corollary}[Theorem]{Corollary}

\newtheorem{Example}[Theorem]{Example}

\newtheorem{Definition}[Theorem]{Definition}

\let\epsilon\varepsilon
\let\phi=\varphi
\let\kappa=\varkappa
\def\qed{\ifhmode\textqed\fi
      \ifmmode\ifinner\quad\qedsymbol\else\dispqed\fi\fi}
\def\textqed{\unskip\nobreak\penalty50
       \hskip2em\hbox{}\nobreak\hfil\qedsymbol
       \parfillskip=0pt \finalhyphendemerits=0}
\def\dispqed{\rlap{\qquad\qedsymbol}}

\opn\dis{dis}
\def\pnt{{\raise0.5mm\hbox{\large\bf.}}}

\opn\Lex{Lex}

\begin{document}

\title{The $\mathcal{N}\mathcal{F}$-Number of a Simplicial Complex}
	\author[T.~Hibi]{Takayuki Hibi}
\address[Takayuki Hibi]{Department of Pure and Applied Mathematics,
	Graduate School of Information Science and Technology,
	Osaka University,
	Suita, Osaka 565-0871, Japan}
\email{hibi@math.sci.osaka-u.ac.jp}
\author[H.~Mahmood]{Hasan Mahmood}
\address[Hasan Mahmood]{Government College University Lahore, Pakistan}
\email{hasanmahmood@gcu.edu.pk}
\maketitle
\begin{abstract}
Let $\Delta$ be a simplicial complex on $[n]$.  The $\mathcal{N}\mathcal{F}$-complex of $\Delta$ is the simplicial complex $\delta_{\mathcal{N}\mathcal{F}}(\Delta)$ on $[n]$ for which the facet ideal of $\Delta$ is equal to the Stanley--Reisner ideal of $\delta_{\mathcal{N}\mathcal{F}}(\Delta)$.  Furthermore, for each $k = 2,3,\ldots$\,, we introduce {\em $k^{th}$ $\mathcal{N}\mathcal{F}$-complex} $\delta^{(k)}_{\mathcal{N}\mathcal{F}}(\Delta)$ which is inductively defined by $\delta^{(k)}_{\mathcal{N}\mathcal{F}}(\Delta) = \delta_{\mathcal{N}\mathcal{F}}(\delta^{(k-1)}_{\mathcal{N}\mathcal{F}}(\Delta))$ with setting $\delta^{(1)}_{\mathcal{N}\mathcal{F}}(\Delta) = \delta_{\mathcal{N}\mathcal{F}}(\Delta)$.  One can set $\delta^{(0)}_{\mathcal{N}\mathcal{F}}(\Delta) = \Delta$.  The $\mathcal{N}\mathcal{F}$-number of $\Delta$ is the smallest integer $k > 0$ for which $\delta^{(k)}_{\mathcal{N}\mathcal{F}}(\Delta) \simeq \Delta$.  In the present paper we are especially interested in the $\mathcal{N}\mathcal{F}$-number of a finite graph, which can be regraded as a simplicial complex of dimension one.  It is shown that the $\mathcal{N}\mathcal{F}$-number of the finite graph $K_n\coprod K_m$ on $[n + m]$, which is the disjoint union of the complete graphs $K_n$ on $[n]$ and $K_m$ on $[m]$, where $n \geq 2$ and $m \geq 2$ with $(n,m) \neq (2,2)$, is equal to $n + m + 2$.  Its corollary says that the $\mathcal{N}\mathcal{F}$-number of the complete bipartite graph $K_{n,m}$ on $[n+m]$ is also equal to $n + m + 2$.
\end{abstract}

\section*{Introduction}
The Stanley--Reisner ideal of a simplicial complex was introduced in 1974 by Stanley \cite{S} and Reisner \cite{R} independently.  On the other hand, Faridi \cite{F} studies the facet ideal of a simplicial complex.  Given a simplicial complex $\Delta$, one can naturally associate a simplicial complex $\delta_{\mathcal{N}\mathcal{F}}(\Delta)$ for which the facet ideal of $\Delta$ coincides with the Stanley--Reisner ideal of $\delta_{\mathcal{N}\mathcal{F}}(\Delta)$.  The topic of the present paper is the sequence of simplicial complexes
\[
\Delta, \, \delta_{\mathcal{N}\mathcal{F}}(\Delta), \, \delta_{\mathcal{N}\mathcal{F}}(\delta_{\mathcal{N}\mathcal{F}}(\Delta)), \, \cdots
\]
arising from $\Delta$.  Fundamental materials together with our main theorem will be presented in Section $1$, while its proof will be given in Section $2$.

\section{Facet ideals and Stanley--Reisner complexes}
A {\em simplicial complex} $\Delta$ on the vertex set $[n] = \{1,\ldots,n\}$ is a collection of subsets of $[n]$ with the property that if $F \in \Delta$ and if $G \subset F$, then $G \in \Delta$.  (We do {\em not} require the property that $\{i\} \in \Delta$ for each $i \in [n]$.)  Each $F \in \Delta$ is called a {\em face} of $\Delta$.  A {\em facet} is a maximal face of $\Delta$.  Let $\mathfrak{F}(\Delta)$ denote the set of facets of $\Delta$. When $\mathfrak{F}(\Delta)=\{F_1, \ldots, F_r\}$, it is clear that $\Delta$ consists of those subsets $G \subset [n]$ for which there is $1 \leq i \leq n$ with $G \subset F_i$.  One can then write $\Delta = \langle F_1, \ldots, F_r \rangle$.  Finally, the {\em dimension} of $\Delta$ is $\dim \Delta = d - 1$, where $d$ is the maximal cardinality of faces of $\Delta$.

Given simplicial complexes $\Delta$ and $\Delta'$ on $[n]$, we say that $\Delta$ is {\em ismorphic} to $\Delta'$ if there is a permutation $\pi$ on $[n]$ with $\pi(\Delta) = \Delta'$, where
$
\pi(\Delta) = \{ \pi(F) : F \in \Delta \}
$
and where
$
\pi(F) = \{ x_{\pi(i)} ; x_i \in F \}.
$

Let $S = K[x_1, \ldots, x_n]$ denote the polynomial ring in $n$ variables over a field $K$.
\begin{itemize}
\item
Given a simplicial complex $\Delta$ on $[n]$, the {\em facet ideal} of $\Delta$ is the ideal $\mathcal{I}_{\mathcal{F}}(\Delta)$ of $S$ generated by those squarefree monomials $\prod_{i \in F}x_i$ with $F \in \mathfrak{F}(\Delta)$.
\item
Given a squarefree monomial ideal $I$ of $S$, the {\em Stanley--Reisner complex} of $I$ is the simplicial complex $\Delta_{\mathcal{N}}(I)$ on $[n]$ consisting of those subsets $F \subset [n]$ for which $\prod_{i \in F} x_i \not\in I$.
\end{itemize}

Thus in particular if $\Delta = \{ \emptyset \}$, then $\mathcal{I}_{\mathcal{F}}(\Delta) = (0)$.  If $I = (0)$, then $\Delta_{\mathcal{N}}(I) = \langle [n] \rangle$.  If $\Delta = \langle [n] \rangle$, then $\mathcal{I}_{\mathcal{F}}(\Delta) = (x_1 \cdots x_n)$.  Furthermore, if $I = (x_1, \ldots, x_n)$, then $\Delta_{\mathcal{N}}(I) = \{ \emptyset \}$.

Let $\Delta$ be a simplicial complex on $[n]$.  A subset $C$ of $[n]$ is said to be a {\em vertex cover} of $\Delta$ if $C \cap F \neq \emptyset$ for each $F \in \mathfrak{F}(\Delta)$.  A vertex cover $C$ of $\Delta$ is {\em minimal} if no proper subset of $C$ forms a vertex cover of $\Delta$.  Let ${\rm MIN}(\Delta)$ denote the set of minimal vertex covers of $\Delta$.
The standard primary decomposition (\cite[p.~12]{HH}) of $\Ic_{\Fc}(\Delta)$ is
\begin{eqnarray}
\label{primary}
\Ic_{\Fc}(\Delta) = \bigcap_{\{x_{i_1}, \ldots, x_{i_s}\} \in {\rm MIN}(\Delta)} (x_{i_1}, \ldots, x_{i_s}).
\end{eqnarray}
Furthermore,
\begin{eqnarray}
\label{facetSR}
\mathfrak{F}(\Delta_{\mathcal{N}}(\Ic_{\Fc}(\Delta))) = \{ F \subset [n] : [n] \setminus F \in {\rm MIN}(\Delta) \}.
\end{eqnarray}

\begin{Example}
\label{aaaaa}
{\em
Let $\Delta$ be the simplicial complex on $[5]$ with $$\Ic_{\Fc}(\Delta) = (x_1x_2, x_2x_3x_4, x_2x_5, x_4x_5).$$
One has
\begin{eqnarray*}
\Ic_{\Fc}(\Delta) & = &
(x_1x_2, x_2x_3x_4, x_2x_5, x_4)
\cap
(x_1x_2, x_2x_3x_4, x_2x_5, x_5) \\
& = &
(x_1x_2, x_2x_5, x_4)
\cap
(x_1x_2, x_2x_3x_4, x_5) \\
& = & (x_2, x_4) \cap (x_2, x_5) \cap (x_1, x_3, x_5) \cap (x_1, x_4, x_5)
\end{eqnarray*}
and
\[
\Delta_{\mathcal{N}}(\Ic_{\Fc}(\Delta)) = \langle \{2,3\}, \{2,4\}, \{1,3,4\}, \{1,3,5\} \rangle.
\]
}
\end{Example}

\begin{Definition}
{\em
Given a simplicial complex $\Delta$ on $[n]$, we say that the Stanley--Reisner complex of the facet ideal of $\Delta$ is the {\em $\mathcal{N}\mathcal{F}$-complex} of $\Delta$.  Let $\delta_{\mathcal{N}\mathcal{F}}(\Delta)$ denote the {\em $\mathcal{N}\mathcal{F}$-complex} of $\Delta$.  Thus
$$
\delta_{\mathcal{N}\mathcal{F}}(\Delta) = \Delta_{\mathcal{N}}(\Ic_{\Fc}(\Delta)).
$$
Furthermore, for each $k = 2,3,\ldots$\,, we introduce {\em $k^{th}$ $\mathcal{N}\mathcal{F}$-complex} $\delta^{(k)}_{\mathcal{N}\mathcal{F}}(\Delta)$ which is inductively defined by $\delta^{(k)}_{\mathcal{N}\mathcal{F}}(\Delta) = \delta_{\mathcal{N}\mathcal{F}}(\delta^{(k-1)}_{\mathcal{N}\mathcal{F}}(\Delta))$ with setting $\delta^{(1)}_{\mathcal{N}\mathcal{F}}(\Delta) = \delta_{\mathcal{N}\mathcal{F}}(\Delta)$.  One can set $\delta^{(0)}_{\mathcal{N}\mathcal{F}}(\Delta) = \Delta$.
}
\end{Definition}

\begin{Example}
{\em
Let $\Delta$ be the simplicial complex of Example \ref{aaaaa}.  Then the standard primary decomposition of $\Ic_\Fc(\delta_{\mathcal{N}\mathcal{F}}(\Delta))$ is
\[
(x_1, x_2, x_4) \cap (x_1, x_2, x_5) \cap (x_2, x_3, x_4) \cap (x_2, x_3, x_5) \cap (x_2, x_4, x_5) \cap (x_3, x_4).
\]
Thus
\[
\delta^{(2)}_{\mathcal{N}\mathcal{F}}(\Delta) = \langle \{1,3\}, \{1,4\}, \{1,5\}, \{3,4\}, \{3,5\}, \{1,2,5\} \rangle.
\]
}
\end{Example}

\begin{Lemma}
\label{NF}
Let $\Delta$ be a simplicial complex on $[n]$.  Then there exist a positive integer $q$ with
$$\delta^{(q)}_{\mathcal{N}\mathcal{F}}(\Delta) = \Delta.$$
\end{Lemma}

\begin{proof}
Suppose that $\delta^{(i)}_{\mathcal{N}\mathcal{F}}(\Delta)
\neq \delta^{(j)}_{\mathcal{N}\mathcal{F}}(\Delta)$
for all $i$ and $j$ with $0 \leq i < j$.  It then turns out that the number of simplicial complexes on $[n]$ cannot be finite.  Clearly, this is a contradiction.  It follows that there exist $0 \leq i < j$ with $\delta^{(i)}_{\mathcal{N}\mathcal{F}}(\Delta) = \delta^{(j)}_{\mathcal{N}\mathcal{F}}(\Delta)$.  Let $j_0$ denote the smallest integer for which there is $0 \leq i < j_0$ with $\delta^{(i)}_{\mathcal{N}\mathcal{F}}(\Delta) = \delta^{(j_0)}_{\mathcal{N}\mathcal{F}}(\Delta)$.  We claim $i = 0$.  Let $i > 0$.  Then $\delta_{\Nc\Fc}(\delta^{(i-1)}_{\mathcal{N}\mathcal{F}}(\Delta)) = \delta_{\Nc\Fc}(\delta^{(j_0-1)}_{\mathcal{N}\mathcal{F}}(\Delta))$.

In general, if $\Delta'$ and $\Delta''$ are simplicial complexes on $[n]$ with $\delta_{\Nc\Fc}(\Delta') = \delta_{\Nc\Fc}(\Delta'')$, then it follows from (\ref{primary}) and (\ref{facetSR}) that the standard primary decomposition of $\Ic_\Fc(\Delta')$ must coincide with that of $\Ic_\Fc(\Delta'')$.  Hence $\Ic_\Fc(\Delta') = \Ic_\Fc(\Delta'')$ and then $\Delta' = \Delta''$.

Now, since $\delta_{\Nc\Fc}(\delta^{(i-1)}_{\mathcal{N}\mathcal{F}}(\Delta)) = \delta_{\Nc\Fc}(\delta^{(j_0-1)}_{\mathcal{N}\mathcal{F}}(\Delta))$, it follows that $\delta^{(i-1)}_{\mathcal{N}\mathcal{F}}(\Delta) = \delta^{(j_0-1)}_{\mathcal{N}\mathcal{F}}(\Delta)$.  This contradict the choice of $j_0$.  Hence $i = 0$, as desired.
\, \, \, \, \, \, \, \, \, \, \, \, \, \,
\end{proof}

Now, Lemma \ref{NF} guarantees the existence of the smallest integer $t_0 \geq 1$ for which $\delta^{(t_0)}_{\mathcal{N}\mathcal{F}}(\Delta)$ is isomorphic to $\Delta$.  The smallest integer $t_0$ is called the {\em $\mathcal{N}\mathcal{F}$-number} of $\Delta$.

\begin{Example}
{\em
Let $n = 3$ and $\Delta = \{ \emptyset \}$.  Then $\delta^{(1)}_{\mathcal{N}\mathcal{F}}(\Delta) = \langle [3] \rangle$.  Thus $\delta^{(2)}_{\mathcal{N}\mathcal{F}}(\Delta) = \langle \{1,2\}, \{2,3\}, \{1,3\} \rangle$ and $\delta^{(3)}_{\mathcal{N}\mathcal{F}}(\Delta) = \langle \{1\}, \{2\}, \{3\} \rangle$.  Hence $\delta^{(4)}_{\mathcal{N}\mathcal{F}}(\Delta) = \Delta$.  Thus the $\mathcal{N}\mathcal{F}$-number of $\Delta = \{ \emptyset \}$ on $[3]$ is $4$.  In general, the $\mathcal{N}\mathcal{F}$-number of $\Delta = \{ \emptyset \}$ on $[n]$ is $n + 1$.
}
\end{Example}

\begin{Example}
{\em
Let $n = 3$ and $\Delta = \{ \{1\}, \{2,3\} \}$.  Then
\begin{eqnarray*}
\delta^{(1)}_{\mathcal{N}\mathcal{F}}(\Delta) & = & \langle \{2\}, \{3\} \rangle, \, \, \, \, \,
\delta^{(2)}_{\mathcal{N}\mathcal{F}}(\Delta) =  \langle \{1\} \rangle, \, \, \, \, \,
\delta^{(3)}_{\mathcal{N}\mathcal{F}}(\Delta) = \langle \{2,3\} \rangle, \\
\delta^{(4)}_{\mathcal{N}\mathcal{F}}(\Delta) & = & \langle \{1,2\}, \{1,3\} \rangle, \, \, \, \, \,
\delta^{(5)}_{\mathcal{N}\mathcal{F}}(\Delta) = \langle \{1\}, \{2, 3\} \rangle.
\end{eqnarray*}
Thus the $\mathcal{N}\mathcal{F}$-number of $\Delta$ is $5$.
}
\end{Example}

\begin{Example}
{\em
Let $\Delta = \{\{1,2\},\{2,3\},\{3,4\}\}$ be the simplicial complex on $[4]$.  Then
$\delta^{(1)}_{\mathcal{N}\mathcal{F}}(\Delta) = \langle \{1,3\},\{1,4\},\{2,4\} \rangle$ and $\delta^{(2)}_{\mathcal{N}\mathcal{F}}(\Delta) =  \Delta$.  Since $\delta^{(1)}_{\mathcal{N}\mathcal{F}}(\Delta)$ is isomorphic to $\Delta$, it follows that the $\mathcal{N}\mathcal{F}$-number of $\Delta$ is equal to $1$.
}
\end{Example}

Even though the $\mathcal{N}\mathcal{F}$-number can be defined for an arbitrary simplicial complex, in the present paper we are especially interested in the $\mathcal{N}\mathcal{F}$-number for a finite graph, which can be regarded as a simplicial complex of dimension one.

Let $P_n$ denote the {\em path} on $[n]$.  Thus the edges of $P_n$ are those $\{i, i+1\}$ with $1 \leq i < n$.  Let $C_n$ denote the {\em cycle} on $[n]$.  Thus the edges of $C_n$ are those $\{i, i+1\}$ with $1 \leq i < n$ together with $\{1, n\}$.  Let $K_n$ denote the {\em complete graph} on $[n]$.  Thus the edges of $K_n$ are those $\{i, j\}$ with $1 \leq i < j \leq n$.

\begin{Example}
{\em
(a) Let $a_n$ denote the $\mathcal{N}\mathcal{F}$-number of $P_n$.  Then
\[
a_2 = 3 , a_3 = 5, a_4 = 1, a_5 = 8, a_6 = 48 , a_7 = 47, a_8 = 552
\]

(b) Let $b_n$ denote the $\mathcal{N}\mathcal{F}$-number of $C_n$.  Then
\[
b_3 = 4 , b_4 = 2, b_5 = 2 , b_6 = 12, b_7 = 8 , b_8 = 26, b_9 = 139
\]

(c) Let $c_n$ denote the $\mathcal{N}\mathcal{F}$-number of $K_n$.  Then
\[
c_n = n+1
\]
}
\end{Example}

Let $K_n\coprod K_m$ denote the finite graph on $[n+m]$, which is the disjoint union of $K_n$ and $K_m$.  In the present paper the $\mathcal{N}\mathcal{F}$-number of $K_n\coprod K_m$ is computed.

\begin{Example}
{\em
The $\mathcal{N}\mathcal{F}$-number of $K_2\coprod K_2$ is equal to $2$.
}
\end{Example}

We now come to the main result of this paper.

\begin{Theorem}
\label{main}
The $\mathcal{N}\mathcal{F}$-number of the finite graph $K_n\coprod K_m$ on $[n+m]$, where $n \geq 2$ and $m \geq 2$ with $(n,m) \neq (2,2)$, is equal to $n + m + 2$.
\end{Theorem}

Let $K_{n,m}$ denote the complete bipartite graph on
\[
[n+m] = \{1,\ldots,n\} \cup \{n+1,\ldots, n+m\}.
\]

\begin{Corollary}
\label{cor}
The $\mathcal{N}\mathcal{F}$-number of the complete bipartite graph $K_{n,m}$ on $[n+m]$, where $n \geq 2$ and $m \geq 2$ with $(n,m) \neq (2,2)$, is equal to $n + m + 2$.
\end{Corollary}

The following Section $2$ is devoted to the proof of Theorem \ref{main}.  Corollary \ref{cor} follows easily from the computation done in the proof of Theorem \ref{main}.

\section{The $\mathcal{N}\mathcal{F}$-Number of $K_n\coprod K_m$}

From now on, we fix integers $n \geq 2$ and $m \geq 2$ with $(n,m) \neq (2,2)$.  Instead of $[n + m]$, the vertex set of $K_n\coprod K_m$ is denoted by $V=V_n \cup V_m$, where $$V_n = \{x_1, \ldots, x_n\}, \, \, \, V_m = \{y_1, \ldots, y_m\}.$$

\begin{Lemma}
\label{first}
Let $\Delta=K_n\coprod K_m$.  One has
\begin{itemize}
\item[(i)]
$\Ff(\delta^{(1)}_{\mathcal{N}\mathcal{F}}(\Delta)) = \{ \{x_i,y_j\} : i \in [n], \, j \in [m] \}$;
\item[(ii)]
$\Ff(\delta^{(2)}_{\mathcal{N}\mathcal{F}}(\Delta)) = \{ V_n, V_m \}$;
\item[(iii)]
$\Ff(\delta^{(3)}_{\mathcal{N}\mathcal{F}}(\Delta)) = \{ (V_n \setminus \{x_i\}) \cup (V_m \setminus \{y_j\}) : i \in [n], \, j \in [m] \}$.
\end{itemize}
\end{Lemma}

\begin{proof}
Let $M \in {\rm MIN}(\Delta)$.  Then $|M\cap V_n|\geq n-1$ and $|M\cap V_m|\geq m-1$.  It then follows that
$${\rm MIN}(\Delta)= \{(V_n \setminus \{x_i\})\cup (V_m \setminus \{y_j\}) : i\in [n], \, j\in [m]\},$$
which further gives (i) by equation (2), as required.

Let $M \in {\rm MIN}(\delta^{(1)}_{\mathcal{N}\mathcal{F}}(\Delta))$. If there is $i \in [n]$ with $x_i \not\in M$.  Then $V_m \subset M$.  If there is $j \in [m]$ with $y_j \not\in M$.  Then $V_n \subset M$.  Since $V_n$ and $V_m$ belong to ${\rm MIN}(\delta^{(1)}_{\mathcal{N}\mathcal{F}}(\Delta))$ the desired (ii) follows.

It is clear that ${\rm MIN}(\delta^{(2)}_{\mathcal{N}\mathcal{F}}(\Delta))$ consists of $\{x_i,y_j\}$ with $i \in [n]$ and $j \in [m]$, which guarantees (iii) as desired.
\, \, \, \, \, \, \, \, \, \, \, \, \, \, \, \, \, \, \, \,
\, \, \, \, \, \, \, \, \, \, \, \, \, \, \, \, \,
\end{proof}

Let $0 \leq i \leq n$ and $0 \leq j \leq m$ be integers.  We then set
$$\mathcal{M}_{(i,j)}=\{ F\subset V_n \cup V_m : |F\cap V_n| = i, |F\cap V_m| = j \}.$$
Thus $|\mathcal{M}_{(i,j)}|={n \choose i}{m \choose j}$.  Furthermore, we set
$$\mathcal{M}^{c}_{(i,j)}=\{V \setminus F: F \in \mathcal{M}_{(i,j)}\}.$$ In particular one has $$\mathcal{M}^{c}_{(i,j)}= \mathcal{M}_{(n-i,m-j)}.$$

For $\Delta=K_n\coprod K_m$ on $V_n \cup V_m$, finding facets of simplicial complexes
$$
\delta^{(k)}_{\mathcal{N}\mathcal{F}}(\Delta), \, \, \, \, \,
4\leq k \leq m+n+2$$
is indispensable for proving Theorem \ref{main}.

\begin{Lemma}{\label{FP}}
Let, as before, $\Delta=K_n\coprod K_m$ on $V_n \cup V_m$.  Suppose that $n\leq m$.
\begin{itemize}
\item[(i)]
If $4\leq k \leq n+2$, then
$\Ff(\delta^{(k)}_{\mathcal{N}\mathcal{F}}(\Delta))$ is equal to
\begin{eqnarray}
\label{H1}
\mathcal{M}^{c}_{(k-2, 0)} \cup \mathcal{M}^{c}_{(0, k-2)} \bigcup_{1\leq i \leq n, \, 1 \leq j \leq m, \, i+j=k-3}\mathcal{M}^{c}_{(i, j)}.
\end{eqnarray}
\item [(ii)]
One has
\begin{eqnarray}
\label{H2}
\Ff(\delta^{(n+3)}_{\mathcal{N}\mathcal{F}}(\Delta)) = \mathcal{M}^{c}_{(0,n+1)} \bigcup_{1\leq i \leq n, \, 1 \leq j \leq m, \, i+j=n}\mathcal{M}^{c}_{(i,j)}.
\end{eqnarray}
\item[(iii)]
If $n + 2 \leq m$ and $n+4\leq k \leq m+2$, then
$\Ff(\delta^{(k)}_{\mathcal{N}\mathcal{F}}(\Delta))$ is equal to
\begin{eqnarray}
\label{H3}
\mathcal{M}^{c}_{(0, k-2)}\cup\mathcal{M}^{c}_{(n, k-4-n)} \cup \bigcup_{1\leq i \leq n - 1, \, 1 \leq j \leq m, \, i+j=k-3}\mathcal{M}^{c}_{(i,j)}.
\end{eqnarray}
\item [(iv)]
If $m+3\leq k \leq m+n+2$, then
$\Ff(\delta^{(k)}_{\mathcal{N}\mathcal{F}}(\Delta))$ is equal to
\begin{eqnarray}
\label{H4}
\mathcal{M}^{c}_{(n, k-4-n)}\cup \mathcal{M}^{c}_{(k-4-m, m)} \cup \bigcup_{1\leq i \leq n - 1, \, 1 \leq j \leq m -1, \, i+j=k-3}\mathcal{M}^{c}_{(i,j)}.
 \end{eqnarray}
\end{itemize}
\end{Lemma}

\begin{proof}
First, we prove (\ref{H1}) by using induction on $k$.  Let $k=4$.  Since $$\bigcup_{1\leq i, \, 1\leq j, \, i+j=1}\mathcal{M}^{c}_{(i,j)} = \emptyset,$$
it is required to show that
\begin{eqnarray}
\label{AAAAA}
\Ff(\delta^{(4)}_{\mathcal{N}\mathcal{F}}(\Delta)) = \mathcal{M}^{c}_{(2,0)} \cup \mathcal{M}^{c}_{(0,2)}.
\end{eqnarray}
Lemma \ref{first} (iii) says that
${\rm MIN}(\delta^{(3)}_{\mathcal{N}\mathcal{F}}(\Delta))$
coincides with
\[
\{ \{x_{i},x_{i'}\} : i\neq i' \}
\cup
\{ \{y_{j},y_{j'}\} : j\neq j' \}.
\]
Thus by using $(\ref{facetSR})$ the desired $(\ref{AAAAA})$ follows.

Now, suppose that the formula $(\ref{H1})$ is valid for a fixed $k$ with $4 \leq k < n+2$.  Then a subset $F \subset V_n \cup V_m$ belongs to $\Ff(\delta^{(k)}_{\mathcal{N}\mathcal{F}}(\Delta))$ if and only if one of the following conditions is satisfied:
\begin{itemize}
\item
$|F \cap V_n| = n - (k - 2)$ and $V_m \subset F$;
\item
$V_n \subset F$ and $|F \cap V_m| = m - (k - 2)$;
\item
$|F \cap V_n| = p < n$, $|F \cap V_m| = q < m$ and $p + q = (n + m) - (k - 3)$.
\end{itemize}
It then follows that $M \subset V_n \cup V_m$ belongs to ${\rm MIN}(\delta^{(k)}_{\mathcal{N}\mathcal{F}}(\Delta))$ if and only if one of the following conditions is satisfied:
\begin{itemize}
\item
$M \subset V_n$ and $|M| = k - 1$;
\item
$M \subset V_m$ and $|M| = k - 1$;
\item
$|M \cap V_n| = p' \geq 1$, $|M \cap V_m| = q' \geq 1$ and $p' + q' = k - 2$.
\end{itemize}
In other words, one has
\[
{\rm MIN}(\delta^{(k)}_{\mathcal{N}\mathcal{F}}(\Delta))=\mathcal{M}_{(k-1,0)} \cup \mathcal{M}_{(0,k-1)} \bigcup_{1\leq i \leq n, \, 1 \leq j \leq m, \, i+j=k-2}\mathcal{M}_{(i,j)},
\]
and the desired $(\ref{H1})$ for $k+1$ follows.

\medskip

Let $k = n + 2$ in $(\ref{H1})$.  One has
\[
\Ff(\delta^{(n+2)}_{\mathcal{N}\mathcal{F}}(\Delta)) = \mathcal{M}^{c}_{(n,0)} \cup \mathcal{M}^{c}_{(0,n)} \bigcup_{1\leq i \leq n, \, 1 \leq j \leq m, \, i+j=n-1}\mathcal{M}^{c}_{(i,j)}.
\]
It then follows that
$${\rm MIN}(\delta^{(n+2)}_{\mathcal{N}\mathcal{F}}(\Delta))= \mathcal{M}_{(0,n+1)} \bigcup_{1\leq i \leq n, \, 1 \leq j \leq m, \, i+j=n}\mathcal{M}_{(i,j)},$$
from which the desired $(\ref{H2})$ follows.

\medskip

On the other hand, by using $(\ref{H2})$, one has
\[
{\rm MIN}(\delta^{(n+3)}_{\mathcal{N}\mathcal{F}}(\Delta)) =
\Mc_{(0, n+2)} \cup \Mc_{(n, 0)} \bigcup_{1 \leq i \leq n - 1, \, 1 \leq j \leq m, \, i + j = n + 1} \Mc_{(i, j)},
\]
from which the desired $(\ref{H3})$ for $k = n + 4$ follows.  Now, the obvious technique for proving $(\ref{H1})$ shows the desired $(\ref{H3})$ for $n+5 \leq k \leq m+2$.

\medskip

Finally, the routine computation as done above easily finishes showing $(\ref{H4})$ as desired. \, \, \, \, \, \, \, \, \, \,
\, \, \, \, \, \, \, \, \, \, \, \, \, \, \, \, \, \, \, \,
\, \, \, \, \, \, \, \, \, \,
\, \, \, \, \, \, \, \, \, \,
\end{proof}

We are now in the position to prove Theorem \ref{main} together with Corollary \ref{cor}.

\begin{proof}[Proof of Theorem \ref{main}]
Let $n\leq m$. Since
\[
i\in[n-1], \, \, \, j\in [m-1], \, \, \, i+j=n+m+2
\]
possesses no solution, it follows from $(\ref{H4})$ that
\[
\Ff(\delta^{(m+n+2)}_{\mathcal{N}\mathcal{F}}(\Delta)) = \mathcal{M}^{c}_{(n,m-2)}\cup \mathcal{M}^{c}_{(n-2,m)} = \Ff(\Delta).
\]
In other words,
\[
\delta^{(m+n+2)}_{\mathcal{N}\mathcal{F}}(\Delta) = \Delta.
\]
To finish our proof, one must show that
\begin{eqnarray}
\label{simeq}
\delta^{(k)}_{\mathcal{N}\mathcal{F}}(\Delta)\not\simeq \Delta,
\, \, \, \, \,
1 \leq k < n + m + 2.
\end{eqnarray}
As $n \leq m$, this means  $3 \leq m$.  Lemma \ref{first} says that
\[
\dim \delta^{(1)}_{\mathcal{N}\mathcal{F}}(\Delta) = 1,
\dim \delta^{(2)}_{\mathcal{N}\mathcal{F}}(\Delta) = m-1,
\dim \delta^{(3)}_{\mathcal{N}\mathcal{F}}(\Delta) =n+m-3
\]
Furthermore, it follows from Lemma \ref{FP} that
\[ \dim \delta^{(k)}_{\mathcal{N}\mathcal{F}}(\Delta)
 =  \left\{
\begin{array}{ll}
      n+m-k+2 & \text{if}\ \, 4\leq k \leq n+2, \\
      m-1 & \text{if}\ \, k=n+3, \\
      n+m-k+3 & \text{if}\ \, n+3 < k < n+m+2 \\
\end{array}
\right. \]
In particular,
\begin{eqnarray}
\label{simeqdim}
\dim \delta^{(k)}_{\mathcal{N}\mathcal{F}}(\Delta) > 1,
\, \, \, \, \,
2 \leq k < k' < n + m + 2.
\end{eqnarray}
Furthermore, since $\delta^{(1)}_{\mathcal{N}\mathcal{F}}(\Delta) \not\simeq \Delta$, the desired (\ref{simeq}) follows.
\, \, \, \, \, \, \, \, \, \, \, \, \,
\end{proof}

\begin{proof}[Proof of Corollary \ref{cor}]
Let $\Delta = K_n\coprod K_m$.  Lemma \ref{first} says that $K_{n,m} = \delta^{(1)}_{\Nc\Fc}(\Delta)$.  Thus the desired result follows from $(\ref{simeqdim})$.
\, \, \, \, \, \, \, \, \, \, \, \, \, \, \,
\, \, \, \, \, \, \, \, \, \,
\end{proof}

We finish this paper with giving a reasonable question.  Let $\Delta$ and $\Gamma$ be simplicial complexes on $[n]$.  We say that $\Delta$ and $\Gamma$ are {\em $\Nc\Fc$-equivalent} if there is $k \geq 0$ with $\Gamma = \delta_{\Nc\Fc}^{(k)}(\Delta)$.  It follows from Lemma \ref{NF} that the $\Nc\Fc$-equivalence is an equivalence relation.  The equivalence class to which $\Delta$ belongs consists of
\[
\delta_{\Nc\Fc}^{(k)}(\Delta), \, \, \, \, \, 0 \leq k \leq q,
\]
where $q$ is the $\Nc\Fc$-number of $\Delta$.  Let $\Nc\Fc(n)$ denote the number of equivalence classes in the set of simplicial complexes on $[n]$.  A reasonable question is to find a combinatorial formula to compute $\Nc\Fc(n)$.

\end{document}